\documentclass[a4paper, 12pt]{article}

\usepackage[a4paper]{geometry}
\usepackage{amsmath, amsfonts, amssymb, amsthm, mathrsfs, mathtools}
\usepackage{nicematrix}
\usepackage{csquotes}
\usepackage{indentfirst}
\usepackage{chngcntr}
\usepackage{enumitem}
\usepackage{titlesec}
\usepackage{caption}
\usepackage{cmap}
\usepackage{floatrow}
\usepackage{algpseudocode}
\usepackage{algorithm}
\usepackage{lmodern}
\usepackage[colorlinks, linkcolor = purple, citecolor = purple]{hyperref}
\usepackage{orcidlink}
\usepackage{diagbox}
\usepackage{xurl}
\usepackage{hyperref}

\mathtoolsset{showonlyrefs=true}
\titleformat{\section}
 {\normalfont\fontsize{16}{16}\bfseries}{\thesection}{1em}{}

\DeclareFontEncoding{LS2}{}{\noaccents@}
\DeclareFontSubstitution{LS2}{stix}{m}{n}
\DeclareSymbolFont{arrows3}{LS2}{stixtt}{m}{n}
\DeclareMathSymbol{\squareulblack}{\mathord}{arrows3}{"88}

\theoremstyle{definition}

\newtheorem{proposition}{Proposition}

\newtheorem{theorem}{Theorem}

\newtheorem{corollary}{Corollary}
\newtheorem*{definition*}{Definition}
\newtheorem*{lemma*}{Lemma}
\newtheorem*{theorem*}{Theorem}
\newtheorem*{corollary*}{Corollary}

\DeclareMathOperator*{\tr}{tr}
\DeclareMathOperator*{\rank}{rank}

\DeclareMathOperator*{\argmin}{argmin}
\DeclareMathOperator*{\argmax}{argmax}

\renewcommand{\geq}{\geqslant}
\renewcommand{\leq}{\leqslant}

\newcommand{\R}{\mathbb{R}}

\newcommand{\vol}{\mathrm{Vol}}

\newcommand{\cS}{\mathcal{S}}

\newcommand{\cT}{\mathcal{T}}
\newcommand{\cI}{\mathcal{I}}
\newcommand{\cM}{\mathcal{M}}

\newcommand{\normf}[1]{\Vert{#1}\Vert_F}
\newcommand{\norms}[1]{\Vert{#1}\Vert_2}

\makeatletter
\newcommand{\MultL}[1]{%
    \parbox[t]{\dimexpr\linewidth-\ALG@tlm-1em}{\raggedright #1\strut}
}
\makeatother

\algblockdefx[Phase]{Phase}{EndPhase}[1]{\textbf{#1}}{}
\algtext*{EndPhase}

\begin{document}

\title{Subset selection for matrices by column exchange}

\author{
  \orcidlink{0000-0002-9956-8223} Alexander Osinsky
  \thanks{Skolkovo Institute of Science and Technology, Moscow, 121205, Russia}
  \thanks{Marchuk Institute of Numerical Mathematics RAS, Moscow, 119333, Russia}\\
  osinskiy1189@gmail.com
  \and
  \orcidlink{0009-0008-3133-064X} Ivan Kozyrev
  \footnotemark[2]
  \thanks{Moscow Institute of Physics and Technology, Dolgoprudny, Moscow Region, 141701, Russia}\\
  kozyrev.in@phystech.edu
}

\maketitle

\begin{abstract}
The paper considers the problem of finding a submatrix $X_{\cS} \in \mathbb{R}^{m \times k}$ in a matrix $X \in \mathbb{R}^{m \times n}$, such that the spectral or Frobenius norm of $X_{\cS}^{\dag} X$ is limited, which guarantees it provides a good representation of the whole matrix. Such bounds can be reached by applying greedy algorithms, maximizing the submatrix volume. We suggest a modification of a greedy volume maximization, which performs column exchanges asymptotically faster for $n \gg m$ than the known alternatives, while guaranteeing the same bounds on $X_{\cS}^{\dag} X$. In addition, we prove a new upper bound on the number of required exchanges, which is applicable to the new algorithm as well as to other greedy volume maximization algorithms.

\vspace{6pt}
    
\noindent \textbf{Keywords:} subset selection,  locally maximum volume, optimal design, sparse approximation \vspace{6pt}
    
\noindent \textbf{AMS subject classifications:} 65F55, 90C27, 15A18, 62K05
\end{abstract}

\section{Introduction}

We study the problem of finding a representative subset $\cS$ of $k$ columns in a matrix $X \in \R^{m \times n}$ ($m \leq k \leq n$, $m \ll n$ in typical applications). I.e., we search for a submatrix $X_\cS \in \R^{m \times k}$. Quantitatively, the quality of the submatrix can be expressed in the size of coefficients, which are needed to express all other columns. In other words, we need to minimize $\normf{X_\cS^\dag X}$, where $X_\cS^\dag$ denotes the Moore-Penrose pseudoinverse of $X_\cS$. 

It is possible to find a submatrix which satisfies
\begin{equation}\label{eqn:matrix_norm}
\normf{X_\cS^\dag X}^2 \leq m + \frac{m + (c^2 - 1)k}{k - m + 1}(n - k)\,,
\end{equation}
or, for each column $x_j$ with $j \notin \cS$,
\begin{equation}\label{eqn:column_norms}
\norms{X_\cS^\dag x_j}^2 \leq \frac{m + (c^2 - 1)k}{k - m + 1}\,,
\end{equation}
with any factor $c \geqslant 1$. Namely, these properties are satisfied for the so-called $c$-locally maximum volume submatrices \cite{osti_276532,Osinsky23}: submatrices, which have at least $1/c$ of the maximum volume $\vol \left( X_\cS \right) = \sqrt{\det \left(X_\cS X_\cS^T \right)}$ among those, that differ from them in a single column. While there are other greedy alternatives, which can guarantee $c = 1$ \cite{faster_subset_selection}, their cost is quadratic in the total number of columns $n$, while volume maximization is linear in $n$.

Note that finding (square) submatrices of maximum volume up to a factor $e^{Cm}$ with a constant $C$ is NP-hard~\cite{VolNPhard}, which is why algorithms search for a locally maximal volume instead. Similar problems like maximizing volume for $k < Cm$, $C<1$ \cite{VolExpRect} up to an exponential in $k$ factor and minimizing $\normf{X_\cS^{\dag}}$ up to a constant factor \cite{doi:10.1287/moor.2021.1129} are also NP-hard. Whether and how these results can be extended to maximization of a volume for $k > m$ or minimization of $\normf{X_\cS^\dag X}$ remain open problems.

When $k=m$ one can use well-known 
Maxvol~\cite{good_submatrix} algorithm, first utilized in~\cite{Mosaic}. To select $k > m$ columns one can greedily add~\cite{MIKHALEV2018187} or exchange~\cite{Osinsky23} columns. Both of these algorithms, however, have their drawbacks. The first of them cannot guarantee equations~\eqref{eqn:matrix_norm}, \eqref{eqn:column_norms} or similar bounds, and the computational cost of both mentioned algorithms becomes too large, when many columns are required. In this paper we solve this problem by providing a fast greedy exchange algorithm for volume maximization, improving the bound on the number of required exchanges, and showing numerically the advantages of the new approach.

For matrices with orthonormal rows, equation~\eqref{eqn:matrix_norm} directly limits spectral and Frobenius norm ratios:
\[
\begin{split}
\frac{\normf{X_\cS^\dag}^2}{\normf{X^\dag}} &= \frac{\normf{X_\cS^\dag X}}{m} \leq \frac{n - m + 1}{k - m + 1} + (c^2 - 1)\frac{k(n - k)}{m(k - m + 1)}\,,\\
\frac{\norms{X_\cS^\dag}^2}{\norms{X^\dag}} &= \norms{X_\cS^\dag X} \leq 1 + \frac{m + (c^2 - 1)k}{k - m + 1}(n - k)\,.
\end{split}
\]
which guarantees, in the worst case~\cite{GOREINOV19971, OSINSKY2018221, Osinsky2023-bg} or on average~\cite{doi:10.1137/140977898, Osinsky2023-bg}, that the cross or column approximations, using these submatrices, have low error, relative to the best approximation, given by truncated SVD. Properties like \eqref{eqn:column_norms} are also useful in optimal design (G-optimality), while volume optimization itself also provides D-optimal designs~\cite{Huan_Jagalur_Marzouk_2024}. Other applications of highly nondegenerate submatrices with bounded $\normf{X_\cS^\dag X}^2$ include DEIM~\cite{deim}, DLRA-DEIM~\cite{carrel2025}, least-squares regression~\cite{BOUTSIDIS2014273}, rank-deficient least squares~\cite{10.1145/1132973.1132981, doi:10.1137/090780882}, preconditioning~\cite{Arioli2015-eb}, graph signal processing~\cite{chen2015discrete, tsitsvero2016signals, ZHANG2025109668}, and the search for low-stretch spanning trees~\cite{faster_subset_selection}.

\subsection{Our contribution}

Our contribution is two-fold. First, we suggest a fast exchange algorithm (Algorithm~\ref{alg:vol_add_rem}) to select rectangular submatrices of large volume. This algorithm guarantees to reach a submatrix $X_\cS$ that represents all other columns with low coefficients, bounded as in \eqref{eqn:column_norms} and thus also guaranteeing
\begin{equation}\label{eqn:norm_bounds}
\begin{split}
\normf{X_\cS^\dag X}^2 &\leq m + \frac{m + (c^2 - 1)k}{k - m + 1}(n - k)\,,\\
\norms{X_\cS^\dag X}^2 &\leq 1 + \frac{m + (c^2 - 1)k}{k - m + 1}(n - k)\,.
\end{split}
\end{equation}
Each exchange consists of two steps: adding and removing a column, for which we provide a fast and stable approach that works asymptotically faster than the greedy addition \cite{MichDoc} or the maxvol-like greedy exchange \cite{Osinsky23}.

Secondly, we propose an advanced initialization procedure that yields a better lower bound on the starting volume. This directly improves the upper bound on the number of required exchanges to reach~\eqref{eqn:norm_bounds}, effectively reducing the overall complexity by removing the logarithmic factor of the order $\log k$.

Together, these two results decrease the overall cost of reaching the bounds \eqref{eqn:norm_bounds} from $O(n k m \log_c k)$ to $O(nkm + n m^2 \log \log m + n m^2 / \log c)$. Our results also provide stronger bounds on the number of exchanges in both Maxvol \cite{good_submatrix} and Dominant \cite{Osinsky23} algorithms.

New methods are verified on both unstructured data (isotropic matrices with i.i.d. Gaussian entries) and highly structured realistic data (transposed right singular vectors obtained from the oriented edge-vertex incidence matrices of random weighted connected graphs).

\subsection{Organization}

The remainder of the paper is organized as follows. Section~\ref{sec:preliminaries} reviews the necessary background, including the column pivoting QR (CPQR) decomposition and the Dominant algorithms. Section~\ref{sec:vol_add_rem} details the proposed Dominant-split algorithm and its theoretical guarantees. Section~\ref{sec:adv_init} outlines the advanced initialization procedure. Section~\ref{sec:experiments} presents numerical experiments evaluating the practical performance of our methods, and Section~\ref{sec:conclusion} concludes the paper.

\section{Preliminaries}\label{sec:preliminaries}

In this section we describe two algorithms directly related to our contributions, and crucial to understanding the paper. 

\subsection{Column pivoted QR decomposition}

The CPQR~\cite[page 278]{doi:10.1137/1.9781421407944} algorithm is a modification of $QR$ decomposition, which not only performs the decomposition itself, but additionally selects a well-conditioned basis from the columns of the input matrix. This algorithm iteratively selects the column which has the largest component orthogonal to already selected ones.

An important consequence for us is that submatrix $X_\cI$ selected by $CPQR$ cannot have volume differing too much from the optimal volume \cite{subset_selection_complexity}. Here we prove a slightly stronger result

\begin{proposition}
$CPQR$ produces a submatrix $X_\cI$, such that
\begin{equation}\label{eqn:cpqr_bound}
\vol\left(X_\cI\right) \geq \frac{1}{m^{m/2}}\max_{\cT \subseteq \overline{1, n}\,,\ \vert\cT\vert = m} \vol\left(X_\cT\right)\,.
\end{equation}
\end{proposition}
\begin{proof}
  Without loss of generality consider $\cI$ to be the first $m$ columns of $X$. Instead of $X$ we can consider the same question in the factor $R \in \mathbb{R}^{m \times n}$ after QR decomposition $X = QR$, $Q \in \mathbb{R}^{m \times m}$, as multiplication by $Q$ does not change the volume. According to column pivoting criterion we have
  \[
    R_{ii}^2 = \max\limits_j \left\| R_{i:m,j} \right\|_2^2 \geqslant \max\limits_j R_{ij}^2,
  \]
  where $R_{i:m,j}$ contains elements of $j$-th column of $R$ from $i$ to $m$.
  
  Thus, summing the squares of all elements, we have $\normf{{\rm diag}^{-1} ( R_{\cI} ) R_{\cT}}^2 \leqslant m^2$, where we multiplied by the inverse of the diagonal of $R_{\cI}$. Then, using AM-GM inequality to bound the volume through Frobenius norm, we obtain
  \[
    \frac{\vol\left(R_\cT\right)}{\vol\left(R_\cI\right)} = \frac{\vol\left({\rm diag}^{-1} ( R_{\cI} )R_\cT\right)}{\vol\left({\rm diag}^{-1} ( R_{\cI} )R_\cI\right)} \leqslant \left( \frac{\left\| {\rm diag}^{-1} ( R_{\cI} ) R_{\cT} \right\|_F^2}{m} \right)^{m/2} \leqslant m^{m/2}.
  \]\qedhere
\end{proof}
\begin{corollary}
  The same bound holds if instead of CPQR we would use Gaussian elimination with row pivoting, as it also guarantees that in LU decomposition each pivot in U has the largest absolute value in its row and multiplication by $L$ does not change the ratio of volumes.
\end{corollary}

Generalising this to the case of $m \times k$ matrices using Cauchy-Binet formula, we have the following corollary.

\begin{corollary}\label{cor:vol_bounds_k}
If subset of column indices $\cS$ of cardinality $k$ contains subset $\cI$ produced by $CPQR$, then
\begin{equation}\label{eqn:vol_bounds_k}
\vol\left(X_\cS\right) \geq \frac{1}{m^{m/2}\sqrt{\binom{k}{m}}}\max_{\cT \subseteq \overline{1, n}\,,\ \vert\cT\vert = k} \vol\left(X_\cT\right)\,.
\end{equation}
\end{corollary}

This bound allows to limit the number of iterations in Dominant algorithm discussed in the next subsection.

\subsection{Dominant algorithm}

The Dominant algorithm~\cite{Osinsky23} is an algorithm for finding a $c$-locally maximum volume submatrix, i.e. a submatrix $X_\cS$ such that
\begin{equation}\label{eqn:c_loc_opt}
c\vol\left(X_\cS\right) \geq \max_{s \notin \cS\,,\ r\in\cS}\vol\left(X_{\cS + s - r}\right)\,.
\end{equation}
Here, the parameter $c \geq 1$, and when $c = 1$ the submatrix is called locally maximum volume submatrix. This submatrix has some desirable properties~\eqref{eqn:matrix_norm},~\eqref{eqn:column_norms}.

Dominant algorithm reaches such submatrix by iteratively updating the current submatrix $X_\cS$ by swapping one index (say, $r \in \cS$) for another ($s \notin \cS$). The removed and the added columns are selected to maximize the volume $\vol(X_{\cS'})$, where $\cS' = \cS + s - r$. The scores for all possible swaps are effectively updated through iterations, allowing $O(nk)$ iteration cost. Once~\eqref{eqn:c_loc_opt} is satisfied, algorithm terminates.

It is easy to see that such algorithm requires finite number of swaps to terminate, as there is only finite number of submatrices, and algorithm goes from submatrices with smaller volume to ones with larger volume. We can however limit the number of iterations much better if $c > 1$ and the initial submatrix contains columns chosen by $CPQR$ thanks to Corollary~\ref{cor:vol_bounds_k}. Indeed, on each iteration before termination the volume of the selected submatrix grows by a factor of $c$, thus there can not be more than
\begin{equation}\label{eqn:bound_excahnges}
\log_{c}\left(m^{m/2}\sqrt{\binom{k}{m}}\right) = \frac{1}{2}\log_c\left(\frac{k!m^m}{(k - m)!m!}\right) \leq \frac{1}{2}\log_c\left(k^m e^m\right) = \frac{1}{2} m\log_c \left( ek \right)\,.
\end{equation}
iterations. Therefore the total asymptotic complexity of the Dominant algorithm is $O(nkm\log_ck)$. Note that in practice the number of swaps required to reach $c$-locally maximum volume submatrix is usually much lower as we demonstrate in Section~\ref{sec:experiments}.

Overall, the Dominant algorithm is an attractive choice for finding a submatrix with large volume or bounded pseudoinverse. With adequate choice of $c$ it provides bounds almost as tight as those of Algorithm~1 in~\cite{faster_subset_selection}, but has a much lower asymptotic complexity, especially for small $k$. However, as $k$ grows, the algorithm requires more and more computations. 

In the succeeding sections we propose two modifications of the dominant algorithm which target two distinct aspects of it: first we show how to make each iteration cheaper without compromising the bounds, and next we address the matter of choosing the initial submatrix.

\section{Splitting the column exchange}\label{sec:vol_add_rem}

In this section we construct a modification of the Dominant algorithm~\cite{Osinsky23}, which has identical guarantees in terms of spectral and Frobenius norm of the pseudoinverse, but has a lower asymptotic complexity per iteration.

New algorithm runs in $O(n m^2 \log_c k)$ time, compared to $O(n k m \log_c k)$ of the Dominant algorithm. This complexity reduction is achieved by splitting the update: on each iteration we first choose the optimal index $s$ to add:
\[
s \in \argmax_{j \notin \cS}\vol(X_{\cS + j})\,,
\]
and then remove the worst one: 
\[
r \in \argmin_{j \in \cS} \vol(X_{\cS + s - j})\,,
\]
thus performing an exchange. Logic for selecting the initial set $\cS$ and termination criterion stay the same; we perform column swaps while they increase the volume of selected submatrix by a factor of $c > 1$.

This technique provides faster updates, as we need to consider $n - k$ candidate columns for selection and then $k$ options for removal instead of $k \times (n - k)$ possible swaps in Dominant algorithm. The bound on number of swaps~\eqref{eqn:bound_excahnges} still applies, as it is only based on the fact that we increase volume $c$ times on each step. As we will also show, such exchanges still reach the desired bounds including \eqref{eqn:column_norms}.

We now proceed with rigorous analysis of the algorithm, starting with details on how to implement column exchange logic as effectively as possible. 

Given the currently selected indices $\cS$ ($\vert \cS \vert = k \geq m$, $\rank X_\cS = m$), we introduce array $l$, where $l_j$ equals squared norm of $j$-th column of $X_\cS^\dag X$:
\begin{equation}\label{eqn:l_def}
l_j = \left\Vert\left(X_\cS^\dag X\right)_j\right\Vert_2^2 =
\left\Vert X_\cS^T\left(X_\cS X_\cS^T\right)^{-1}x_j\right\Vert_2^2 = x_j^T\left(X_\cS X_\cS^T\right)^{-1}x_j\,.
\end{equation}
Coefficients of $l$ have specific meaning in the context of column exchange. Indeed, suppose $s \notin \cS$. Then, by matrix-determinant lemma, 
\begin{equation}\label{eqn:vol_change_append}
    \vol^2(X_{\cS + s}) = \left(1 + x_s^T\left(X_\cS X_\cS^T\right)^{-1}x_s\right)\vol^2(X_\cS) = (1 + l_s)\vol^2(X_\cS)\,.
\end{equation}
Similarly, for $r \in \cS$ we obtain
\begin{equation}\label{eqn:vol_change_remove}
\vol^2(X_{\cS - r}) = \left(1 - x_r^T\left(X_\cS X_\cS^T\right)^{-1}x_r\right)\vol^2(X_\cS) = (1 - l_r)\vol^2(X_\cS)\,.
\end{equation}

Moreover, we can effectively update $l$ once the new column is added to or removed from $\cS$ as shown in the following proposition. Compared to addition criterion from \cite{MichDoc}, we do the updates based on the matrix $Y = (X_\cS X_\cS^T)^{-1}$.

\begin{proposition}\label{stt:vol_crit}
    Fix $X \in \R^{m \times n}$ ($m \leq n$, $\rank X = m$), $\cS \subseteq \overline{1, n}$ ($\vert \cS \vert = k$, $\rank X_\cS = m$), $Y = (X_\cS X_\cS^T)^{-1}$, and let $l$ be defined as per~\eqref{eqn:l_def}. If a new index $s \notin \cS$ is added to $\cS$, $l'$ corresponding to $\cS + s$ satisfies
    \begin{equation}\label{eqn:l_update_selection}
    l_j' = l_j - \frac{\left(x_s^T Y x_j\right)^2}{1 + l_s}\,.
    \end{equation}
    And if index $r \in \cS$ is removed from $\cS$, $l'$ corresponding to $\cS - r$ satisfies
    \begin{equation}\label{eqn:l_update_removal}
    l_j' = l_j + \frac{\left(x_r^T Y x_j\right)^2}{1 - l_r}\,.
    \end{equation}
\end{proposition}

\begin{proof}
    The proof follows immediately from Sherman-Morrison formula. Indeed, suppose that a new index $s$ was added to $\cS$. Then, 
    \[
    l'_j = x_j^T\left(X_{\cS + s} X_{\cS + s}^T\right)^{-1}x_j = x_j^T\left(Y - \frac{Yx_sx_s^TY}{1 + x_s^TYx_s}\right)x_j = l_j - \frac{\left(x_s^T Y x_j\right)^2}{1 + l_s}\,.
    \]

    The proof for the case of index removal is completely analogous.
\end{proof}

Also note, that since we search for volume, instead of $X$ we can use any matrix $AX$ ($A \in \R^{m \times m}$, $\det(A) \neq 0$), as it does not affect the volume ratio of its submatrices:
\[
\vol \left(\left(AX\right)_\cS\right) = \vol(A)\vol \left(X_\cS\right)\,.
\]
In particular, if we use $A = L^{-1}$ from LQ decomposition of $X$, we do not have numerical problems coming from poor condition number of $Y$, as we usually start from a well conditioned submatrix after CPQR. One can also use $A = X_{\cS}$ fixed at the point, when $\left| \cS \right| = m$, which also limits the condition number of $AX$ and thus maximum condition number of $Y$ (which starts at $Y = (( AX)_{\cS} (AX)_{\cS}^T )^{-1} = I$).

We now provide the pseudocode for the resulting algorithm, denoted as Algorithm~\ref{alg:vol_add_rem}. The properties of this algorithm, including guarantees on the result and asymptotic complexity are stated in Theorem~\ref{thm:vol_add_rem}.

\begin{algorithm}[htbp]
\caption{Dominant-split.}\label{alg:vol_add_rem}
\begin{algorithmic}[1]
\Require $X \in \R^{m \times n}$ ($m \leq n$, $\rank X = m$), starting set of indices $\cS \subseteq \overline{1, n}$ ($\vert\cS\vert = k \geq m$, $\rank X_\cS = m$), parameter $c \geq 1$.
\Ensure Subset of column indices $\cS \subseteq \overline{1,n}$ with $|\cS| = k$.

\State Compute LQ decomposition of $X$, assign $X \gets Q$\footnote{Optional, but good for numerical stability.}.
\State Initialize $Y \gets (X_\cS X_\cS^T)^{-1}$ and array $l$, where $l_j \gets x_j^T Y x_j$ for $j \in \overline{1, n}$. 
\Loop
    \State Select $s \in \argmax_{j \notin \cS} l_j$.
    \State Compute entries of array $l'$, where $l'_j \gets l_j - (x_s^T Y x_j)^2 / (1 + l_s)$ for $j \in \overline{1, n}$. 
    \State Set $Y' \gets Y - Yx_sx_s^TY / (1 + l_s)$.
    \State Select $r \in \argmin_{j \in \cS} l'_j$.
    \If{$(1 + l_s)(1 - l'_r) \leq c^2$}
        \State\textbf{break}
    \EndIf
    \State Update $\cS$: $\cS \gets \cS + s - r$.
    \State Update $l$: $l_j \gets l'_j + (x_rY'x_j)^2 / (1 - l'_r)$ for $j \in \overline{1, n}$. 
    \State Update $Y$: $Y \gets Y' + Y'x_rx_rY'/(1 - l'_r)$.
\EndLoop
\State \textbf{return} $\cS$

\end{algorithmic}
\end{algorithm}

\begin{theorem}\label{thm:vol_add_rem}
    Let matrix $X$, initial subset of indices and parameter $c$ satisfy input requirements of the Algorithm~\ref{alg:vol_add_rem}. Then, set $\cS$ returned by the Algorithm~\ref{alg:vol_add_rem}, satisfies
    \[
    \begin{split}
    \normf{X_\cS^\dag X}^2 &\leq m + \frac{m + (c^2 - 1)k}{k - m + 1}(n - k)\,,\\
    \norms{X_\cS^\dag X}^2 &\leq 1 + \frac{m + (c^2 - 1)k}{k - m + 1}(n - k)\,.
    \end{split}
    \]

    Additionally, if initial subset of indices contains set $\cI$ of cardinality $m$, chosen by $CPQR$ and $c > 1$, Algorithm~\ref{alg:vol_add_rem} requires no more than $O(nm^2\log_ck)$ operations.
    
\end{theorem}

\begin{proof}
\textbf{Proof of bounds.} First, observe that throughout the iterations, array $l$ is maintained to match~\eqref{eqn:l_def}, and $l'$ is array corresponding to set $\cS + s$. Indeed, update formulas in the algorithm match~\eqref{eqn:l_update_selection} and~\eqref{eqn:l_update_removal}, and matrices $Y$ and $Y'$ are updated according to Sherman-Morrison formula to match $(X_\cS X_\cS^T)^{-1}$ and $(X_{\cS + s} X_{\cS + s}^T)^{-1}$ respectively. Also, replacing $X$ with its orthonormal factor $Q$ does not affect the exchange process, as it does not affect the volume ratios for submatrices, as we mentioned earlier.

Now, suppose that algorithm outputted set $\cS$, as the selected pair of indices failed to increase the volume by a factor of $c$. In terms of $l$ and $l'$ it translates to $(1 + l_s)(1 - l'_r) \leq c^2$, see~\eqref{eqn:vol_change_append},~\eqref{eqn:vol_change_remove}. And since $r$ provides minimum to $l'_j$, $j \in S$, and $(1 + l_s)(1 - l'_s) = 1$, we have
\begin{equation}\label{eqn:col_norm_bound}
\begin{split}
\frac{kc^2 + 1}{1 + l_s} &\geq \sum_{j \in \cS + s}(1 - l'_j) = \sum_{j \in \cS + s} \left(1 - x_j^T\left(X_{\cS + s} X_{\cS + s}^T\right)^{-1}x_j\right)\\
&= k + 1 - \tr\left(X_{\cS + s}^T\left(X_{\cS + s} X_{\cS + s}^T\right)^{-1}X_{\cS + s}\right) = k + 1 - m\,.
\end{split}
\end{equation}
Using the above and noting that $l_s$ provides maximum to $l_j$, $j \notin \cS$ we have
\begin{equation}\label{eqn:ljadd}
\begin{split}
\frac{(n - k)(kc^2 + 1)}{k - m + 1} &\geq \sum_{j \notin \cS} (1 + l_j) = \sum_{j \notin \cS} \left(1 + x_j^T\left(X_\cS X_\cS^T\right)^{-1}x_j\right)\\ 
&= n - k - m +\tr\left(X^T\left(X_\cS X_\cS^T\right)^{-1}X\right) = n - k - m + \normf{X_\cS^\dag X}^2\,.
\end{split}
\end{equation}
Transforming this expression, we obtain
\[
\normf{X_\cS^\dag X}^2 \leq m\frac{n - m + 1}{k - m + 1} + \frac{(c^2 - 1)k(n - k)}{k - m + 1}\,.
\]
The bound \eqref{eqn:column_norms} is similarly obtained by considering each column individually in \eqref{eqn:ljadd}.

Note that $\normf{X_\cS^\dag X}^2$ depends only on orthonormal factor of $X$, and it does not matter whether we have performed the $QR$ decomposition in the beginning of the algorithm or not.

The spectral norm bound follows from the Frobenius norm bound written as
\[
\sum_{i = 1}^m \sigma_i^{2}(X_\cS^\dag X) \leq C(m,k,n)\,.
\]
Moving terms $i=2,\dots,m$ to the right and noting that $\sigma_i(X_\cS^\dag X) \geq 1$ we obtain
\[
\sigma_i^{2}(X_\cS^\dag X) \leq C(m,k,n) - m + 1 = 1 + \frac{m + (c^2 - 1)k}{k - m + 1}(n - k)\,.
\]

\textbf{Proof of complexity.} Both usual and rank-revealing $QR$ decompositions cost $O(nm^2)$. Initializing $Y$ costs $O(m^3)$, computing $l$ requires $O(nm^2)$ operations.

After that each column exchange iteration requires $O(nm)$ iterations to compute $l'$ and update $l$ and $O(m^2)$ operations to compute $Y'$ and update $Y'$, resulting in $O(nm)$ operations in total. Number of swaps is limited by $2m\log_c k$ (See Corollary~\ref{cor:vol_bounds_k} and~\eqref{eqn:bound_excahnges}), resulting in $O(nm^2 + nm^2\log_c k)$ complexity of the Algorithm~\ref{alg:vol_add_rem}.
\end{proof}

\section{Advanced initialization}\label{sec:adv_init}

The number of steps/exchanges for the update procedure heavily depends on the initial ratio between the volume of the starting submatrix and the maximum volume. When the initial columns are selected at random, we expect to do much more steps, than when the columns were all added greedily. Indeed, as we will show, we can often completely get rid of (or reduce) the logarithmic factor in complexity. To do that, we first establish a connection between the volume of submatrices produced by Algorithm~\ref{alg:vol_add_rem} and the maximum volume submatrices, generalizing a similar result from~\cite{mearxiv,Maxvol2bound}.

\begin{theorem}\label{thm:locvol_maxvol}
Let matrix $X$, subset $\cS$ of cardinality $k$ and parameter $c$ satisfy input requirements of Algorithm~\ref{alg:vol_add_rem}, and be such that algorithm exits immediately, without performing any swaps. Then,
\begin{equation}\label{eqn:maxvol2-cons_res}
\vol \left(X_\cS\right) \geq \left( \frac{k+1}{k-m+1} \left( 1 + \left( c^2 - 1 \right) \frac{k}{m} \right) \right)^{-\frac{m}{2}} \max_{\cT \subseteq \overline{1, n}\,,\ \vert\cT\vert = k} \vol\left(X_\cT\right)\,.
\end{equation}

Additionally, if $m + (c^2 - 1)m \geq k - m + 1$ (in particular, if $k \leqslant 2m-1$), then
\begin{equation}\label{eqn:maxvol2-cons_res2}
\vol \left(X_\cS\right) \geq \left( \frac{k}{k-m+1} \left( 1 + \left( c^2 - 1 \right) \frac{k}{m} \right) \right)^{-\frac{m}{2}} \max_{\cT \subseteq \overline{1, n}\,,\ \vert\cT\vert = k} \vol\left(X_\cT\right)\,.
\end{equation}
\end{theorem}
\begin{proof}
\begin{enumerate}[wide]
\item First, observe that when Algorithm~\ref{alg:vol_add_rem} terminates, we have
\begin{equation}\label{eqn:minlj_bound}
  \left(1 - \min_{j \in \cS} l'_j \right)\left(1 + \max_{j \notin \cS} l_j \right) \leq c^2\,.
\end{equation}
Using the averaging argument on above, we can obtain~\eqref{eqn:col_norm_bound}. Rearranging the terms in this equation results in 
\begin{equation}\label{eqn:maxli_bound}
  \max_{j \notin \cS}l_j \leq \frac{m + (c^2-1)k}{k-m+1}\,.
\end{equation}
The rest of the proof is based on these two facts.

\item Now, let $m + (c^2 - 1)k \geq k - m + 1$. In that case right-hand side in~\eqref{eqn:maxli_bound} is greater then or equal to $1$. Also note that $l_j$ for $j \in \cS$ are norms of the columns in $X_\cS^\dag X_\cS$, which results in $l_j \leq 1$. Therefore~\eqref{eqn:maxli_bound} can be extended to all $l_j$ in this case:
\begin{equation}\label{eqn:x_bound}
\max_{j \in \overline{1, n}}l_j \leq \frac{m + (c^2-1)k}{k-m+1}\,.
\end{equation}

Now, suppose that maximum volume submatrix is given by column indices in $\cM$. Then for this submatrix
\begin{equation}\label{eqn:frob_bound_1}
\normf{X_\cS^\dag X_\cM}^2 = \sum_{j \in \cM}l_j \leq k\frac{m + (c^2-1)k}{k-m+1}\,.
\end{equation}
Using the AM-GM inequality to bound the volume through the Frobenius norm we obtain
\begin{equation}\label{eqn:vol_bound_frob}
\begin{aligned}
\frac{\vol(X_\cM)}{\vol(X_\cS)} & = \vol(X_\cS^\dag X_\cM) \\
& \leq \left(\frac{\normf{X_\cS^\dag X_\cM}^2}{m}\right)^{\frac{m}{2}} = \left( \frac{k}{k-m+1} \left( 1 + \left( c^2 - 1 \right) \frac{k}{m} \right) \right)^{\frac{m}{2}}\,.
\end{aligned}
\end{equation}

\item In the case $m + (c^2 - 1)k < k - m + 1$, the above reasoning holds only if $\cS$ and $\cM$ have no common indices, which is not always the case. However we can bound the Frobenius norm of $X_\cM$ even in this case.  

Observe that $l'_j \leq l_j$ for all $j \in \overline{1, n}$, thus 
\[
\left(1 - \min_{j \in \cS} l_j \right)\left(1 + \max_{j \notin \cS} l_j \right) \leq \left(1 - \min_{j \in \cS} l'_j \right)\left(1 + \max_{j \notin \cS} l_j \right) \leq c^2\,.
\]
Denoting $\max_{j \notin \cS} l_j = x$, we have $l_j \geq \max\{0,\ 1 - c^2/(1 + x)\}$ for all $j \in \cS$. 

Suppose $\cS$ and $\cM$ have $r \geq 1$ common indices. Then, combining the above bound with the fact that $\normf{X_\cS^\dag X_\cS}^2 = \sum_{j \in \cS} l_j = m$, we obtain
\[
\sum_{j \in \cS \cap \cM}l_j = m - \sum_{j \in \cS \setminus \cM} l_j \leq m - (k - r)\max\left\{0, 1 - \frac{c^2}{1 + x}\right\}
\]
and
\begin{equation}\label{eqn:pinv_bound_x}
\normf{X_\cS^\dag X_\cM}^2 = \sum_{j \in \cM \setminus \cS}l_j + \sum_{j \in \cM \cap \cS}l_j \leq m + (k - r)\left(x -\max\left\{0, 1 - \frac{c^2}{1 + x}\right\}\right)\,.
\end{equation}

To bound the right-hand side of~\eqref{eqn:pinv_bound_x} we note that it is maximized when $r = 1$. Studying its dependence on $x$, we obtain that if $1 + x \leq c^2$, $\normf{X_\cS^\dag X_\cM}^2 \leq m + (k - r)x \leq m + (k - 1)(c^2 - 1)$, which is even better then~\eqref{eqn:frob_bound_1}. Otherwise, observe that right-hand side of~\eqref{eqn:pinv_bound_x} has positive derivative in $x$, and thus bounded by its value on the right boundary~\eqref{eqn:x_bound}, which produces
\begin{equation}\label{eqn:frob_bound_tp1}
\begin{aligned}
\normf{X_\cS^\dag X_\cM}^2 &\leq m + \frac{(k - 1)(m + (c^2 - 1)k)}{k - m + 1} + \frac{(k - 1)(c^2(m - 1) + 1)}{c^2k + 1}\\
&\leq m\frac{k+1}{k-m+1} \left( 1 + \left( c^2 - 1 \right) \frac{k}{m} \right) + R\,,
\end{aligned}
\end{equation}
where remainder equals
\[
R = m - 2\frac{m + (c^2-1)k}{k - m + 1} + \frac{(k - 1)(c^2(m - 1) + 1)}{c^2k + 1}\,.
\]

To prove the theorem we must demonstrate that $R \leq 0$. To do so, we introduce $d = k - m + 1$. Eliminating $m$ from $R$ using newly introduced variable and simplifying we obtain
\[
\begin{aligned}
R &= k - d + 1 - 2\frac{c^2k + 1 -d}{d} + \frac{(k - 1)(c^2(k - d) + 1)}{c^2k + 1}\\
&= 4 - 2\frac{c^2k + 1}{d} - d\frac{c^2 + 1}{c^2k + 1} \overset{(*)}{\leq} 4 - \sqrt{2(c^2 + 1)} \leq 0\,,
\end{aligned}
\]
$(*)$ follows from AM-GM inequality. The final bound on $\vol(X_\cM) / \vol(X_\cS)$ now follows from~\eqref{eqn:vol_bound_frob} applied to the newly obtained bound~\eqref{eqn:frob_bound_tp1} where $R \leq 0$.\qedhere
\end{enumerate}
\end{proof}

For some values of $k$ (for instance $k = 2m - 1$) this bound allows for finer estimates of $\vol(X_\cS)$ in Algorithm~\ref{alg:vol_add_rem} than the simple observation that $\vol(X_\cS)$ must increase by $c$ times on each iteration. We will split our result into two parts: for $k \leq 2m-1$ and $k > 2m-1$. For small $k$ we can immediately get an improvement using the following proposition.

\begin{proposition}\label{stt:2r-1}
Fix $X \in \R^{m \times n}$ ($m \leq n$, $\rank X = m$) and let $k \in \overline{m,2m - 1}$. Then in $O(nm^2 \log \log m + nm^2 / \log c)$ arithmetic operations it is possible to find an $m \times k$ submatrix $X_\cS$ satisfying 
\[
\vol \left(X_\cS\right) \geq \frac{c^{-m}}{\sqrt{\binom{2m - 1}{m}}}\max_{\cT \subseteq \overline{1, n}\,,\ \vert\cT\vert = k} \vol\left(X_\cT\right)\,.
\]
\end{proposition}

\begin{proof} 
\begin{enumerate}[wide]
\item First, we use Theorem~\ref{thm:locvol_maxvol} to derive new kind of guarantees on volume of the selected submatrix. Suppose the Algorithm~\ref{alg:vol_add_rem} is running, and $V_t$ is the ratio of the volumes of currently selected submatrix relative to the maximum volume submatrix on iteration $t$. Denote 
\[
\left(c'_t\right)^2 = (1 + l_s)(1 - l'_r) = V_{t + 1}^2 / V_{t}^2\,.
\]
Then we can apply Theorem~\ref{thm:locvol_maxvol} for $c'_t$ to obtain
\[
V_{t} \geq \left(\frac{2m - 1}{m}\left(1 + \left(\frac{V_{t + 1}^2}{V_{t}^2} - 1\right)\frac{2m-1}{m} \right)\right)^{-\frac{m}{2}}\,.
\]
Transforming the above, we have
\begin{equation}\label{eqn:V_tV_t+1}
\frac{m}{2m - 1}V_t^{-\frac{2}{m}} \leq 1 + \left(\frac{V_{t + 1}^2}{V_t^2} - 1\right)\frac{2m - 1}{m} \leq \frac{2m - 1}{m}\frac{V_{t + 1}^2}{V_t^2}\,.
\end{equation}
Isolating $V_{t + 1}$ on the left-hand side, we get $V_{t + 1} \geq mV_{t}^{1-1/m}/(2m -1)$. Setting 
\[
V'_t = V_t \left(\frac{m}{2m - 1}\right)^{m}\,
\]
we obtain 
\begin{equation}\label{eqn:l_prime_recursion}
V_{t + 1}' \geq \left(V_{t}'\right)^{1 - \frac{1}{m}}\,.
\end{equation}

\item\label{itm:step_1_2r-1} Finding the desired submatrix is done in three steps. On the first step we run Algorithm~\ref{alg:vol_add_rem} with initial subset $\cS$ of size $2m - 1$ containing a subset identified by a CPQR algorithm. By Corollary~\ref{cor:vol_bounds_k} we have
\[
V'_0 \geq \left(\frac{m}{2m - 1}\right)^{m} \frac{1}{m^{m/2}\sqrt{\binom{2m - 1}{m}}} \geq \left(\frac{1}{4\sqrt{m}}\right)^{m}\,.
\]
Suppose we make swaps until $V'_t \geq e^{-m}$. According to~\eqref{eqn:l_prime_recursion}, it is enough to do $t$ iterations, where $t$ satisfies
\[
\left(\frac{1}{4 \sqrt{m}}\right)^{m \left(1 - \frac{1}{m}\right)^t} \geq e^{-m}\quad \Leftrightarrow\quad \left(1 - \frac{1}{m}\right)^t \leq \log^{-1}\left(4 \sqrt{m}\right)\,.
\]
Using the inequality $(1 - 1/m)^m \leq e^{-1/m}$, it is enough for $t$ to satisfy $t \geq m \log \log \left( 4 \sqrt{m} \right)$. Since each swap costs $O(nm)$ operations, this phase requires $O(nm^2 \log\log m)$.

\item\label{itm:step_2_2r-1} After the first phase, we have a submatrix of size $m \times (2m - 1)$ with $V'_t \geq e^{-m}$. Our target is to achieve $l_j \leq c^2$ for all $j$, as it will guarantee $\vol(X_\cS) / \vol(X_\cM) \leq c^m$ where $X_\cM$ is the maximal volume submatrix of size $m \times m$. This inequality will allow us to bound the volume of $m \times k$ submatrix obtained on the next phase.

By~\eqref{eqn:maxli_bound}, to guarantee $l_j \leq c^2$ we can run Algorithm~\ref{alg:vol_add_rem} with
\[
\left(c'\right)^2 = 1 + (c^2 - 1)\frac{m}{2m - 1}\,.
\]
This will take no more than
\[
\log_{1 + (c^2 - 1)\frac{m}{2m - 1}}\left(e\frac{2m - 1}{m}\right)^m = O\left(\frac{m}{\log c}\right)
\]
operations.

\item Finally, to obtain the $m \times k$ submatrix (recall $k \leq 2m - 1$), we iteratively remove $(2m - 1) - k$ columns from the $m \times (2m-1)$ submatrix, each time maximizing the volume of the remaining submatrix. By generalized Cauchy-Binet formula (Equation~2.1 in~\cite{derezinski2018reverse}),
\begin{equation}\label{eqn:gen_cauchy_binet}
\vol^2\left(X_\cS\right) = \binom{\vert\cS\vert - m}{\vert\cS\vert - m - 1}^{-1} \sum_{\substack{|\cT| = \vert\cS\vert - 1\,,\ \cT \subseteq \cS}} \det\left(X_\cT X_\cT^T\right)\,.
\end{equation}
Therefore when greedily removing one column, we will reduce the volume by no more then $(\vert\cS\vert - m) / \vert\cS\vert$ times. Combining all factors, we obtain
\[
\vol \left(X_\cS\right) \geq c^{-m}\sqrt{\binom{k}{m}/\binom{2m - 1}{m}}\max_{\cT \subseteq \overline{1, n}\,,\ \vert\cT\vert = m} \vol\left(X_\cT\right)\,.
\]
Using~\eqref{eqn:vol_bounds_k} we can formulate the above in terms of $m \times k$ submatrices of maximal volume, which is exactly what stated in the proposition. Greedy column removal is $O(nm)$ per iteration, thus the total complexity is $O(nm^2 \log \log m + nm^2 / \log c)$.\qedhere
\end{enumerate}
\end{proof}

Of course, one can prove analogous proposition with intermediate number of columns different from $2m - 1$ and potentially even better results, but that will not change the asymptotic complexity.

Now we prove similar proposition for $k \geq 2m - 1$.

\begin{proposition}\label{stt:expvol}
Fix $X \in \R^{m \times n}$ ($m \leq n$, $\rank X = m$) and let $k \in \overline{2m - 1,n}$. Then in $O(nm^2 \log \log m + nkm)$ arithmetic operations it is possible to find an $m \times k$ submatrix $X_\cS$ satisfying 
\[
\vol \left(X_\cS\right) \geq 6^{-m/2}\max_{\cT \subseteq \overline{1, n}\,,\ \vert\cT\vert = k} \vol\left(X_\cT\right)\,.
\]
\end{proposition}

\begin{proof}
We start with submatrix obtained using Proposition~\ref{stt:2r-1} with $c^2 = 3$ and number of selected columns $2m - 1$. That is, an $m \times (2m - 1)$ submatrix $X_\cS$ with volume satisfying
\[
V_{0} \geq \frac{c^{-m}}{\sqrt{\binom{2m - 1}{m}}} \geq \frac{3^{-\frac{m}{2}}}{2^{\frac{m}{2}}} = 6^{-\frac{m}{2}}\,.
\]
Here, as in proposition~\ref{stt:2r-1}, $V_t$ denotes the ratio of the volumes of submatrix selected on iteration $t$ relative to the maximum volume submatrix of same size. Such a submatrix can be obtained in no more than $O(nm^2)$ operations.

We then iteratively add columns, greedily maximizing volume on each step. Then there are two options.

\begin{enumerate}[wide]
\item\label{itm:l_geq_frac} Suppose that at each iteration $t$ we have $\max_{j \notin \cS} l_j \geq m/(\vert\cS\vert-m+1)$. Then, the squared volume of selected submatrix grows by $(1 + l_j) = (\vert\cS\vert + 1)/(\vert\cS\vert - m + 1)$ upon selection of a new column $j$~\eqref{eqn:vol_change_append}.

At the same time, by generalised Cauchy-Binet formula~\eqref{eqn:gen_cauchy_binet} ratio of volumes of maximal volume submatrices of size $m \times \vert\cS\vert$ and $m \times (\vert\cS\vert + 1)$ can not exceed $(\vert\cS\vert + 1)/(\vert\cS\vert - m + 1)$. Therefore $V_{t + 1} \geq V_t$, and the final ratio remains at least $6^{-m/2}$.

\item Now, suppose that on some of the iterations $\max_{j \notin \cS} l_j < m/(\vert\cS\vert-m+1)$. Consider the last iteration $t < k - (2m - 1)$ where the above holds. We then can use Theorem~\ref{thm:locvol_maxvol} with parameter $(c')^2 = (\vert\cS\vert + 1) / (\vert\cS\vert - m + 1)$ (indeed, this guarantees $(1 - \min_{j \in \cS} l'_j)(1 + \max_{j \notin \cS} l_j) \leq 1 + \max_{j \notin \cS} l_j \leq (c')^2$), which states
\[
V_t \geq \left( \frac{\vert\cS\vert+1}{\vert\cS\vert-m+1} \left( 1 + \frac{\vert\cS\vert}{\vert\cS\vert - m + 1}\right) \right)^{-\frac{m}{2}} \geq \left(\frac{2m}{m}\left(1 + \frac{2m - 1}{m}\right)\right)^{-\frac{m}{2}} \geq 6^{-\frac{m}{2}}\,.
\]
For all subsequent iterations, as per Item~\ref{itm:l_geq_frac}, the ratio of current volume to the maximal volume will be at least $6^{-m/2}$.
\end{enumerate}

Regarding computational complexity, obtaining initial $m \times (2m - 1)$ submatrix using the method from Proposition~\ref{stt:2r-1} is $O(nm^2\log m)$ ($c = 3$ is fixed in our case). After that we perform $k - (2m - 1)$ column selections. Each of them can be performed in $O(nm)$ operations using~\eqref{eqn:l_update_selection}. Thus the total asymptotic complexity is $O(nkm)$.
\end{proof}

As a consequence, when Algorithm~\ref{alg:vol_add_rem} is run after such initialization, it requires $O \left( nm^2 / \log c \right)$ operations. We now combine the results of Propositions~\ref{stt:2r-1} and~\ref{stt:expvol} into a single initialization procedure, summarized as Algorithm~\ref{alg:adv_init}, and provide a Theorem~\ref{thm:adv_init}.

\begin{algorithm}[htbp]
\caption{Advanced initialization.}\label{alg:adv_init}
\begin{algorithmic}[1]
\Require $X \in \R^{m \times n}$ ($m \leq n$, $\rank X = m$), sampling parameter $k \in \overline{m, n}$, parameter $c \geq 1$.
\Ensure Subset of column indices $\cS \subseteq \overline{1,n}$ with $|\cS| = k$.

\State Run $CPQR$ on $X$ to obtain index set $\cI$ of cardinality $m$.
\State\label{itm:exchanges} Run Algorithm~\ref{alg:vol_add_rem} on $X$ with initial index set $\cS \supseteq \cI$ of cardinality $2m - 1$ and $c^2 = \min(e,\ 1 + 2m/(2m - 1))$. Save obtained $\cS$.

\State Initialize $Y \gets (X_\cS X_\cS^T)^{-1}$ and array $l$, where $l_j \gets x_j^T Y x_j$ for $j \in \overline{1, n}$. 

\If{$k \leq 2m - 1$}
    \State \MultL{%
    Greedily remove $2m - 1 - k$ columns from $\cS$, each time removing $r \in \argmin_{j \in \cS} l_j$, updating $Y$ according to Sherman-Morrison formula and $l$ according to~\eqref{eqn:l_update_removal}.
    }
\Else
    \State \MultL{%
    Greedily append $k - (2m - 1)$ columns, each time adding $s \in \argmax_{j \notin \cS} l_j$, updating $Y$ according to Sherman-Morrison formula and $l$ according to~\eqref{eqn:l_update_selection}.
    }
\EndIf
\State \textbf{return} $\cS$
\end{algorithmic}
\end{algorithm}

\begin{theorem}\label{thm:adv_init}
Let matrix $X$, sampling parameter $k$, and parameter $c > 1$ satisfy input requirements of Algorithm~\ref{alg:vol_add_rem}. Then, if Algorithm~\ref{alg:vol_add_rem} (or Dominant algorithm~\cite{Osinsky23}) is called with the set $\cS$ produced by Algorithm~\ref{alg:adv_init}, it requires no more than $O(nm^2 / \log c)$ (or $O(nkm / \log c)$ for Dominant) total arithmetic operations. 
\end{theorem}

\begin{proof}
Note that Algorithm~\ref{alg:adv_init} implements the methods of obtaining highly nondegenerate submatrices from Proposition~\ref{stt:2r-1} with $c = 3$ and Proposition~\ref{stt:expvol}. Indeed, State~\ref{itm:exchanges} represents Items~\ref{itm:step_1_2r-1},~\ref{itm:step_2_2r-1} in the proof of Proposition~\ref{stt:2r-1} with $c = 3$. The stopping criterion in the first item is $V'_t \geq e^{-m}$, which is guaranteed to hold~\eqref{eqn:V_tV_t+1} if $V'_{t + 1} / V_t < e$. In the second phase, $c^2 = 1 + 2m/(2m - 1)$. To satisfy both stopping criteria, we must take minimum $c$. All other states copy the proof of Propositions~\ref{stt:2r-1},~\ref{stt:expvol} directly.

In case $k \leq 2m - 1$ Propositions~\ref{stt:2r-1} states
\[
\frac{\vol \left(X_\cS\right)}{\vol(X_\cM)} \geq \frac{3^{-\frac{m}{2}}}{\sqrt{\binom{2m - 1}{m}}} \geq 6^{-\frac{m}{2}}\,,
\]
where $X_\cM$ is maximum volume submatrix of size $m \times k$. If $k > 2m - 1$ we have bound with $6^{-m/2}$ directly (Proposition~\ref{stt:expvol}).

When Algorithm~\ref{alg:vol_add_rem} or Dominant algorithm run afterwards, they require no more than $\log_c 6^{m/2} = O(m / \log c)$ iterations, which results in $O(nm^2 / \log c)$ and $O(nkm / \log c)$ complexities, as one iteration costs $O(nm)$ for Algorithm~\ref{alg:vol_add_rem} and $O(nk)$ for Dominant algorithm. 
\end{proof}

Note that while advanced initialization reduces asymptotic complexity of Algorithm~\ref{alg:vol_add_rem} and Dominant algorithm, roughly speaking removing a factor $\log m$, it may not be always more efficient on practice for $k < 2m - 1$, as overhead from selecting and removing additional $2m - 1 - k$ columns can be larger then benefit of performing less swaps. This issue, along with others, is investigated in the succeeding section.

\section{Numerical experiments}\label{sec:experiments}

In this section we describe numerical experiments, targeted at answering the following two questions:
\begin{enumerate}
    \item \textbf{Target functional:} While Algorithm~\ref{alg:vol_add_rem} and Dominant algorithm have identical theoretical guarantees, submatrices they identify are different. Is there a significant difference in performance of algorithms in terms of $\normf{X_\cS^\dag X}$? Can advanced initialization provide better performance?
    \item \textbf{Swap count:} How Algorithm~\ref{alg:vol_add_rem} and Dominant algorithm compare in terms of number of performed swaps? Can advanced initialization reduce that count and provide speed-up on practice?
\end{enumerate}

Algorithms that search for highly nondegenerate submatrices are implemented in C++ using the Eigen library for matrix operations and common decompositions. For plotting and visualization we utilized Matplotlib Python. The complete codebase is available at \url{https://github.com/KozyrevIN/subset-selection-for-matrices}.

In all experiments we compared six exchange-based strategies from the Table~\ref{tbl:alg_summary}, along with well known Algorithm~2 in~\cite{faster_subset_selection} (Frob-removal-orth), which costs $O(nm^2 + nm(n - k))$. In all cases we set $c = 1$.

In Table~\ref{tbl:alg_summary} we summarize the algorithm versions. CPQR initialization consists of running CPQR and using the first $k$ indices from the permuted index set (first $m$ elements in this set are pivots chosen by CPQR) as a starting value of $\cS$. Greedy initialization consists of identifying first $m$ pivot column using CPQR and then choosing $k - m$ additional columns greedily maximizing volume on each step.

\begin{table}[htbp]
    \centering
    \renewcommand{\arraystretch}{1.5}
    \begin{tabular}{|l|c|c|}
        \hline
        \diagbox[width=5cm]{Initialization}{Exchange} & Dominant~\cite{Osinsky23} & Algorithm~\ref{alg:vol_add_rem} \\ 
        \hline
        CPQR & Dominant-CPQR & Dominant-split-CPQR \\ 
        \hline
        greedy & Dominant-greedy & Dominant-split-greedy \\ 
        \hline
        Algorithm~\ref{alg:adv_init} & Dominant-advanced & Dominant-split-advanced \\ 
        \hline
    \end{tabular}
    \caption{Algorithm names generated by combining exchange and initialization strategies.}
    \label{tbl:alg_summary}
\end{table}

\subsection{Experimental methodology}

In our experiments we used the following two types of matrices:

\begin{enumerate}
    \item Matrices with i.i.d. elements sampled from Gaussian distribution with $\mu = 0$, $\sigma = 1$. Maximum-entropy and rotationally invariant properties of those matrices provide a strictly unstructured, isotropic test case.
    
    \item For the second test case we have chosen matrices related to graph theory. Specifically, we use a matrix of transposed right singular vectors obtained from the oriented edge-vertex incidence matrix of a random weighted connected graph. Both graph topology and edge weights were sampled uniformly. This benchmark, on the contrary to the first one, involves highly structured matrices. For $k = m$, minimizing $\normf{X_\cS^\dag X}$ is equivalent to finding a low-stretch spanning tree~\cite{faster_subset_selection}. 
\end{enumerate}

We fixed the matrix size at $100 \times 5000$ with $k$ ranging from $100$ to $300$. For each value of $k$ we generated $64$ matrices and applied all algorithms mentioned above, measuring $\normf{X_\cS^\dag X}$ and number of swaps performed by the algorithm before it reaches the termination criterion.

\subsection{Comparison in terms of the target functional}

The first observation we made comes down to the fact that for large $k$, the difference in performance of algorithms is negligible. Comparing the means of $1/\normf{X_\cS^\dag X}$ across $64$ trials for various algorithms, we obtain that the ratio between best score minus worst score divided by the mean score does not exceed $5.24\%$ for $k \geq 110$, $2.8\%$ for $k \geq 120$ and $2.08\%$ for $k \geq 130$.

For $k \leq 110$ ($\leq 1.1m$) the difference in performance is significant and illustrated on the Figure~\ref{fig:tests_frob_norm}.

\begin{figure}[htbp]
\centering
\includegraphics[width=\textwidth]{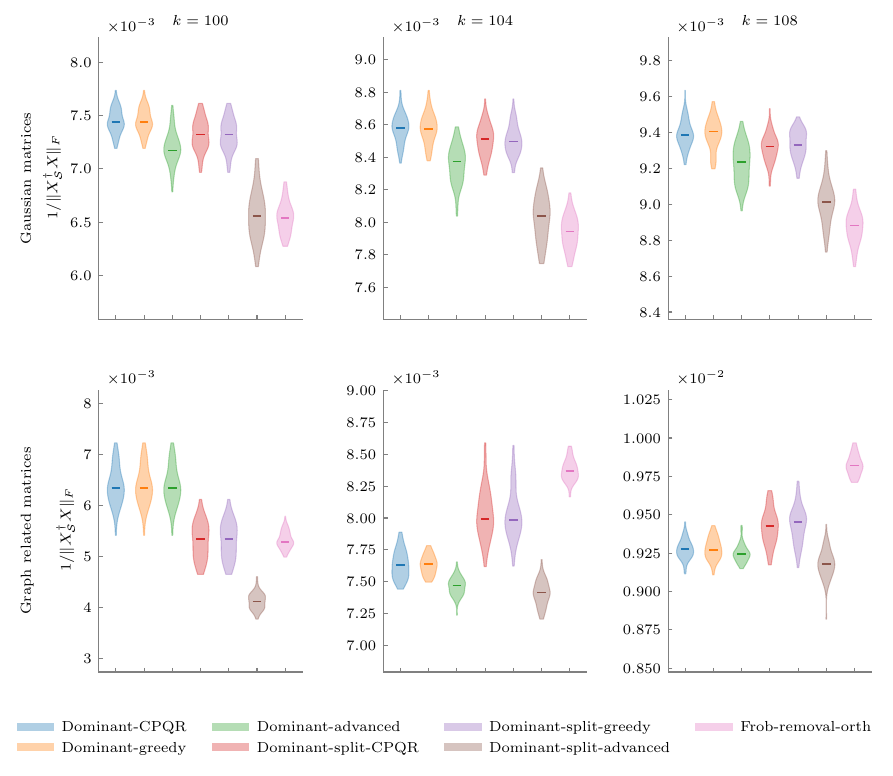}
\caption{Comparison between performance of various algorithms for small values of $k$.}
\label{fig:tests_frob_norm}
\end{figure}

Based on the Figure~\ref{fig:tests_frob_norm}, we can conclude that there is no single winner among the tested algorithms, the performance of algorithms heavily depends on the matrix type and selected number of columns. Next are more concrete observations
\begin{enumerate}
    \item Dominant-split algorithm performs comparably to Dominant algorithm, even though it generally has a lower complexity.
    \item Advanced initialization tends to make results worse both for Dominant and Algorithm~\ref{alg:vol_add_rem}. This is counter-intuitive: advanced initialization provides more degrees of freedom and stronger guarantees on the initial submatrix than other methods. We conclude, that it only needs to be utilized, when CPQR produces submatrix with more than exponentially low volume (compared to, e.g., average expected volume), which never happened.
\end{enumerate}

\subsection{Comparison in terms of swap count}

The comparison between algorithms in terms of number of performed swaps is presented on Figure~\ref{fig:tests_swap_count}.

\begin{figure}[htbp]
\centering
\includegraphics[width=\textwidth]{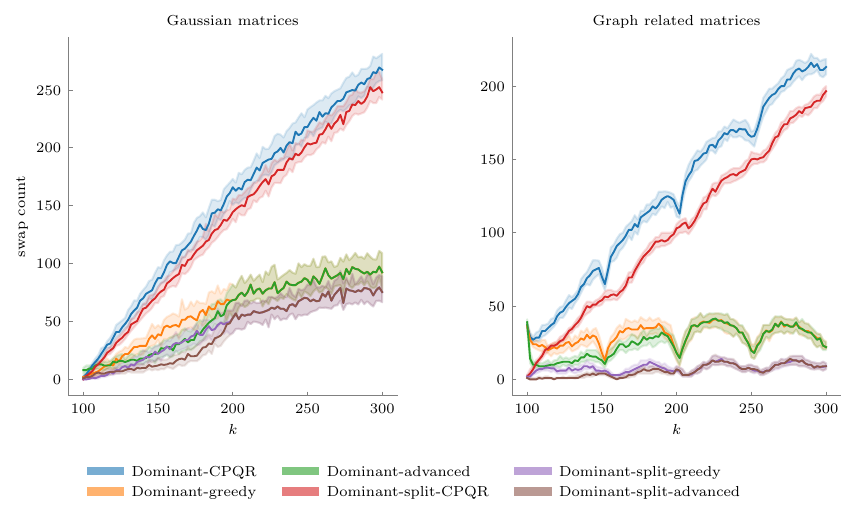}
\caption{Comparison between numbers of swaps performed by various algorithms for small values of $k$. Here, solid line denotes the mean value and shaded regions denote standard deviation.}
\label{fig:tests_swap_count}
\end{figure}

The conclusions we can draw from Figure~\ref{fig:tests_swap_count} are the following:
\begin{enumerate}
    \item Either greedy or advanced initialization is a requisite on practice for any relatively large $k$. Those methods reduce the required swap count dramatically (by orders of magnitude for graph-related matrices). Note that this provides performance benefit, as selecting or removing a column based on the greedy maximization of volume is at least two times cheaper than performing a swap. We also observe that the amount of swaps after greedy and advanced approaches for both algorithms is almost always limited by $m$ independent of $k$.

    \item Dominant algorithm generally requires more swaps then Dominant-split, independently of the initialization type. Note also that each iteration of the Dominant algorithm is dependent on $k$. This makes Dominant-split algorithm favourable for large $k$.

    \item An additional interesting observation is that swap count has visible dips around $k = 150, 200,\dots$. For such values of $k$, it is possible to construct a subgraph with almost all vertices having the same degree, and subset selection algorithms indeed identify sub graphs with this property. We hypothesize that the optima produced by such highly symmetric subgraphs are more pronounced and easier to identify. 
\end{enumerate}

\section{Conclusion}\label{sec:conclusion}

In this article, we proposed two modifications to the Dominant algorithm~\cite{Osinsky23} for subset selection.

First, Algorithm~\ref{alg:vol_add_rem} (Dominant-split) reduces the per-iteration cost from $O(nkm)$ to $O(nm^2)$ while preserving identical theoretical guarantees. Numerical experiments confirm that it also requires fewer swaps and performs comparably in terms of $\normf{X_\cS^\dag X}$.

Second, Algorithm~\ref{alg:adv_init} provides an advanced initialization with stronger volume guarantees than CPQR, reducing the total complexity to $O(nm^2\log\log m + nkm/\log c)$ for Dominant and $O(nm^2\log\log m + nm^2/\log c)$ for Algorithm~\ref{alg:vol_add_rem}. However, experiments show that this initialization can lead to worse $\normf{X_\cS^\dag X}$ when $k \approx m$.

Based on the numerical experiments, we recommend the following practical guidelines. For large $k$, Dominant-split with greedy initialization should be preferred, as it combines low per-iteration cost with significantly fewer swaps. Advanced initialization is beneficial only when CPQR produces a submatrix with unusually low volume; in typical settings, greedy initialization suffices.

\section*{Declarations}

\textbf{Funding.} The research was funded by the Russian Science Foundation (project No. 25-21-00159).

\textbf{Conflict of interests.} The authors declare that they have no conflicts of interest.

\textbf{Compliance with ethical standards.} This article does not contain any studies with human participants or animals performed by any of the authors.

\textbf{Data availability.} Full C++ code used to perform numerical experiments is openly available at \url{https://github.com/KozyrevIN/subset-selection-for-matrices}.

\textbf{Author contributions.} A.O. conceived the original ideas and proofs. I.K. finalized the proofs and performed numerical experiments. Both authors participated in preparing the article.

\bibliographystyle{unsrt}

\bibliography{refs}

@InProceedings{VolNPhard,
  title = 	 {On largest volume simplices and sub-determinants},
  author =       {Di Summa, M. and Eisenbrand, F. and Faenza, Y. and Moldenhauer, C.},
  booktitle = 	 {Proceedings of the Twenty-Sixth Annual ACM-SIAM Symposium on Discrete Algorithms },
  pages = 	 {315--323},
  year = 	 {2015},
  doi = {10.1137/1.9781611973730.23},
  address = {Philadelphia, PA},
  publisher =    {SIAM},
}

@article{MichDoc,
  title={Rectangular submatrices of maximum volume and their computation},
  author={Mikhalev, A. and Oseledets, I.},
  journal={Doklady Mathematics },
  year={2015},
  volume={91},
  pages={267-268},
}

@article{Mosaic,
title = {Mosaic-Skeleton approximations},
journal = {Calcolo},
volume = {33},
pages = {47–57},
year = {1996},
doi = {10.1007/BF02575706},
author = {Tyrtyshnikov, E. E.},
}

@article{VolExpRect,
title = {Exponential Inapproximability of Selecting a Maximum Volume Sub-matrix},
journal = {Algorithmica},
volume = {65},
pages = {159–176},
year = {2013},
doi = {10.1007/s00453-011-9582-6},
author = {Çivril, A. and Magdon-Ismail, M.},
}

@InProceedings{Maxvol2bound,
  title = 	 {Combinatorial Algorithms for Optimal Design},
  author =       {Madan, V. and Singh, M. and Tantipongpipat, U. and Xie, W.},
  booktitle = 	 {Proceedings of the Thirty-Second Conference on Learning Theory },
  pages = 	 {2210--2258},
  year = 	 {2019},
  editor = 	 {Beygelzimer, Alina and Hsu, Daniel},
  volume = 	 {99},
  series = 	 {Proceedings of Machine Learning Research},
  publisher =    {PMLR},
}

@article{mearxiv,
  title = "Rectangular maximum volume and projective volume search algorithms",
  journal = "arXiv 1809.02334 (Submitted on 7 Sep 2018) ",
  author = "Osinsky, A.",
  year = "2018",
}

@article{faster_subset_selection,
author = {Avron, Haim and Boutsidis, Christos},
title = {Faster Subset Selection for Matrices and Applications},
journal = {SIAM Journal on Matrix Analysis and Applications},
volume = {34},
number = {4},
pages = {1464-1499},
year = {2013},
doi = {10.1137/120867287},

URL = {https://doi.org/10.1137/120867287},
eprint = {https://doi.org/10.1137/120867287}
}

@article{subset_selection_complexity,
title = {On selecting a maximum volume sub-matrix of a matrix and related problems},
journal = {Theoretical Computer Science},
volume = {410},
number = {47},
pages = {4801-4811},
year = {2009},
issn = {0304-3975},
doi = {https://doi.org/10.1016/j.tcs.2009.06.018},
url = {https://www.sciencedirect.com/science/article/pii/S0304397509004101},
author = {Ali Çivril and Malik Magdon-Ismail}
}

@article{BOUTSIDIS2014273,
title = {A note on sparse least-squares regression},
journal = {Information Processing Letters},
volume = {114},
number = {5},
pages = {273-276},
year = {2014},
issn = {0020-0190},
doi = {https://doi.org/10.1016/j.ipl.2013.11.011},
url = {https://www.sciencedirect.com/science/article/pii/S0020019013002913},
author = {Christos Boutsidis and Malik Magdon-Ismail}
}

@article{10.1145/1132973.1132981,
author = {Foster, Leslie and Kommu, Rajesh},
title = {Algorithm 853: An efficient algorithm for solving rank-deficient least squares problems},
year = {2006},
issue_date = {March 2006},
publisher = {Association for Computing Machinery},
address = {New York, NY, USA},
volume = {32},
number = {1},
issn = {0098-3500},
url = {https://doi.org/10.1145/1132973.1132981},
doi = {10.1145/1132973.1132981},
journal = {ACM Trans. Math. Softw.},
month = mar,
pages = {157–165},
numpages = {9}
}

@article{doi:10.1137/090780882,
author = {Ipsen, I. C. F. and Kelley, C. T. and Pope, S. R.},
title = {Rank-Deficient Nonlinear Least Squares Problems and Subset Selection},
journal = {SIAM Journal on Numerical Analysis},
volume = {49},
number = {3},
pages = {1244-1266},
year = {2011},
doi = {10.1137/090780882},

URL = { 
    
        https://doi.org/10.1137/090780882
    
    

},
eprint = { 
    
        https://doi.org/10.1137/090780882
    
    

}
}

@article{OSINSKY2018221,
title = {Pseudo-skeleton approximations with better accuracy estimates},
journal = {Linear Algebra and its Applications},
volume = {537},
pages = {221-249},
year = {2018},
issn = {0024-3795},
doi = {https://doi.org/10.1016/j.laa.2017.09.032},
url = {https://www.sciencedirect.com/science/article/pii/S0024379517305669},
author = {A. I. Osinsky and N. L. Zamarashkin}
}

@article{Osinsky2023-bg,
  title     = "Lower bounds for column matrix approximations",
  author    = "A. I. Osinsky",
  journal   = "Computational Mathematics and Mathematical Physics",
  publisher = "Pleiades Publishing Ltd",
  volume    =  63,
  number    =  11,
  pages     = "2024-2037",
  month     =  nov,
  year      =  2023,
  copyright = "https://www.springernature.com/gp/researchers/text-and-data-mining",
  language  = "en"
}

@article{Arioli2015-eb,
  title     = "Preconditioning linear least-squares problems by identifying a
               basis matrix",
  author    = "Arioli, Mario and Duff, Iain S",
  journal   = "SIAM Journal of Scientific Computing",
  publisher = "Society for Industrial \& Applied Mathematics (SIAM)",
  volume    =  37,
  number    =  5,
  pages     = "S544--S561",
  month     =  jan,
  year      =  2015
}

@article{chen2015discrete,
  title={Discrete signal processing on graphs: Sampling theory.},
  author={Chen, Siheng and Varma, Rohan and Sandryhaila, Aliaksei and Kova{\v{c}}evi{\'c}, Jelena},
  journal={IEEE transactions on signal processing},
  volume={63},
  number={24},
  pages={6510--6523},
  year={2015},
  publisher={IEEE}
}

@article{doi:10.1137/140977898,
author = {Boutsidis, Christos and Woodruff, David P.},
title = {Optimal CUR Matrix Decompositions},
journal = {SIAM Journal on Computing},
volume = {46},
number = {2},
pages = {543-589},
year = {2017},
doi = {10.1137/140977898},

URL = { 
    
        https://doi.org/10.1137/140977898
    
    

},
eprint = { 
    
        https://doi.org/10.1137/140977898
    
    

}
}

@article{GOREINOV19971,
title = {A theory of pseudoskeleton approximations},
journal = {Linear Algebra and its Applications},
volume = {261},
number = {1},
pages = {1-21},
year = {1997},
issn = {0024-3795},
doi = {https://doi.org/10.1016/S0024-3795(96)00301-1},
url = {https://www.sciencedirect.com/science/article/pii/S0024379596003011},
author = {S.A. Goreinov and E.E. Tyrtyshnikov and N.L. Zamarashkin}
}

@article{tsitsvero2016signals,
  title={Signals on graphs: Uncertainty principle and sampling},
  author={Tsitsvero, Mikhail and Barbarossa, Sergio and Di Lorenzo, Paolo},
  journal={IEEE Transactions on Signal Processing},
  volume={64},
  number={18},
  pages={4845--4860},
  year={2016},
  publisher={IEEE}
}

@article{Osinsky23,
author = {A. I. Osinsky},
title = {Volume-based subset selection},
journal = {Numerical Linear Algebra with Applications},
volume = {31},
number = {1},
pages = {e2525},
doi = {https://doi.org/10.1002/nla.2525},
url = {https://onlinelibrary.wiley.com/doi/abs/10.1002/nla.2525},
eprint = {https://onlinelibrary.wiley.com/doi/pdf/10.1002/nla.2525},
year = {2024}
}

@article{doi:10.1287/moor.2021.1129,
author = {Nikolov, Aleksandar and Singh, Mohit and Tantipongpipat, Uthaipon (Tao)},
title = {Proportional Volume Sampling and Approximation Algorithms for A-Optimal Design},
journal = {Mathematics of Operations Research},
volume = {47},
number = {2},
pages = {847-877},
year = {2022},
doi = {10.1287/moor.2021.1129},

URL = { 
    
        https://doi.org/10.1287/moor.2021.1129
    
    

},
eprint = { 
    
        https://doi.org/10.1287/moor.2021.1129
    
    

}
}

@article{Huan_Jagalur_Marzouk_2024, title={Optimal experimental design: Formulations and computations}, volume={33}, DOI={10.1017/S0962492924000023}, journal={Acta Numerica}, author={Huan, Xun and Jagalur, Jayanth and Marzouk, Youssef}, year={2024}, pages={715–840}}

@article{osti_276532,
  author       = {Gu, M and Eisenstat, S C},
  title        = {Efficient algorithms for computing a strong rank-revealing QR factorization},
  doi          = {10.1137/0917055},
  url          = {https://www.osti.gov/biblio/276532},
  journal      = {SIAM Journal on Scientific Computing},
  issn         = {ISSN SJOCE3},
  number       = {4},
  volume       = {17},
  place        = {United States},
  year         = {1996},
  month        = {07}}

@article{MIKHALEV2018187,
title = {Rectangular maximum-volume submatrices and their applications},
journal = {Linear Algebra and its Applications},
volume = {538},
pages = {187-211},
year = {2018},
issn = {0024-3795},
doi = {https://doi.org/10.1016/j.laa.2017.10.014},
url = {https://www.sciencedirect.com/science/article/pii/S0024379517305931},
author = {A. Mikhalev and I.V. Oseledets}
}

@article{ZHANG2025109668,
title = {Discrete linear canonical transform on graphs: Uncertainty principle and sampling},
journal = {Signal Processing},
volume = {226},
pages = {109668},
year = {2025},
issn = {0165-1684},
doi = {https://doi.org/10.1016/j.sigpro.2024.109668},
url = {https://www.sciencedirect.com/science/article/pii/S0165168424002883},
author = {Yu Zhang and Bing-Zhao Li}
}

@misc{carrel2025,
      title={Interpolatory Dynamical Low-Rank Approximation: Theoretical Foundations and Algorithms}, 
      author={Benjamin Carrel and Daniel Kressner and Hei Yin Lam and Bart Vandereycken},
      year={2025},
      eprint={2510.19518},
      archivePrefix={arXiv},
      primaryClass={math.NA},
      url={https://arxiv.org/abs/2510.19518}, 
}

@article{deim,
author = {Drma\v{c}, Zlatko and Gugercin, Serkan},
title = {A New Selection Operator for the Discrete Empirical Interpolation Method---Improved A Priori Error Bound and Extensions},
journal = {SIAM Journal on Scientific Computing},
volume = {38},
number = {2},
pages = {A631-A648},
year = {2016},
doi = {10.1137/15M1019271},

URL = { 
    
        https://doi.org/10.1137/15M1019271
    
    

},
eprint = { 
    
        https://doi.org/10.1137/15M1019271
    
    

}
}

@article{derezinski2018reverse,
  title={Reverse iterative volume sampling for linear regression},
  author={Derezi{\'n}ski, Micha{\l} and Warmuth, Manfred K},
  journal={Journal of Machine Learning Research},
  volume={19},
  number={23},
  pages={1--39},
  year={2018}
}

@book{doi:10.1137/1.9781421407944,
author = {Golub, Gene H. and Van Loan, Charles F.},
title = {Matrix Computations - 4th Edition},
publisher = {Johns Hopkins University Press},
year = {2013},
doi = {10.1137/1.9781421407944},
address = {Philadelphia, PA},
edition   = {},
URL = {https://epubs.siam.org/doi/abs/10.1137/1.9781421407944},
eprint = {https://epubs.siam.org/doi/pdf/10.1137/1.9781421407944}
}

@inbook{good_submatrix,
author = {S. A. Goreinov and I. V. Oseledets and D. V. Savostyanov and E. E. Tyrtyshnikov and N. L. Zamarashkin},
title = {How to Find a Good Submatrix},
booktitle = {Matrix Methods: Theory, Algorithms and Applications},
publisher = "World Scientific Publishing",
pages = {247-256},
year={2010},
doi = {10.1142/9789812836021_0015},
URL = {https://www.worldscientific.com/doi/abs/10.1142/9789812836021_0015},
eprint = {https://www.worldscientific.com/doi/pdf/10.1142/9789812836021_0015}
}

\end{document}